\documentclass[11pt]{article}

\usepackage{amsfonts}
\usepackage{mathrsfs}
\usepackage{amsmath}
\usepackage{amsthm}
\usepackage{amssymb}
\usepackage{graphicx}
\usepackage{latexsym}
\usepackage{indentfirst}
\usepackage{authblk}
\usepackage[pagewise]{lineno}

\numberwithin{equation}{section}

\newtheorem{thm}{Theorem}[section]
\newtheorem{lema}[thm]{Lemma}

\newtheorem{rmk}[thm]{Remarks}

\newtheorem{definition}[thm]{Definition}
\newtheorem{theorem}[thm]{Theorem}

\title{\Large\bf Global Solutions to a Fourth-order Degenerate Model for Surface-Tension-Driven Convection}

\author{\sc  Wending Wu\thanks{Corresponding author. E-mail address: wuwending@mail.bnu.edu.cn} and Xiaojing Xu\\
	{\footnotesize \it Laboratory of Mathematics and Complex Systems (Ministry of Education),}\\
	{\footnotesize \it School of Mathematical Sciences, Beijing Normal University, Beijing 100875, China.}}

\date{}

\begin{document}

\maketitle
\noindent {\bf Abstract.} This paper investigates the global existence and non-negativity of weak solutions to an initial-boundary value problem for a one-dimensional fourth-order nonlinear degenerate parabolic equation. This model governs the convection phenomena in thin films driven by surface tension. 
Our analytical approach begins with the formulation of a regularized problem and a corresponding Galerkin approximating scheme. We first establish the existence of solutions to the approximate problem. Subsequently, by constructing specialized energy and entropy functionals, we derive uniform a priori estimates for the approximating solutions. 
Leveraging the Aubin-Lions compactness lemma, we pass to the limit and establish the non-negativity of the limit function. Finally, we demonstrate that this limit is indeed a global weak solution to the original initial-boundary value problem.
\\[0.2cm]
\\[0.2cm]

\section{Introduction}
The emergence of spontaneous ordered structures from homogeneous states is a fundamental characteristic of nonlinear non-equilibrium systems. In fluid dynamics, thermal convection serves as a classical paradigm for studying such self-organization, originating from B\'enard's pioneering experiments. Convection in shallow fluid layers is primarily governed by two competing mechanisms: the bulk buoyancy driving Rayleigh B\'enard convection and the surface tension gradients responsible for Marangoni B\'enard convection. While the former possesses a well-established mathematical framework, the latter dominates in thin films and microgravity environments, presenting a highly nonlinear and mathematically challenging landscape. The intrinsic complexity of Marangoni B\'enard convection arises from the driving force acting directly at the deformable free surface, strongly coupling hydrodynamic stresses with thermal fields \cite{colinet2001nonlinear}. This intricate coupling makes it not only a crucial subject for rigorous mathematical analysis but also a critical element in technologies such as microfluidics and coating processes.

When the free surface is allowed to deform significantly, a new class of instabilities emerges, characterized by long-wave modulations. In this regime, the full three-dimensional Navier-Stokes equations coupled with heat transfer and a dynamic boundary condition can be systematically reduced via long-wave asymptotic methods. The resultant evolution equation, often a nonlinear partial differential equation of high order, captures the essential physics in a more tractable form. A canonical model for a thin liquid layer heated from below, originally derived by Davis \cite{davis} and further discussed in \cite{colinet2001nonlinear}, governs the evolution of the film height $h(t,x)$:
\begin{align}\label{laiyuan}
	h_t=-\nabla\cdot\left( \frac{1}{3Ca}h^3\nabla\Delta h-\frac{Ga}{3}h^3\nabla h+ \frac{Ma}{2}h^2\nabla h\right),
\end{align}
where $t>0$ is time and $x \in \mathbb{R}^2$ is the spatial coordinate. The operators $\nabla$ and $\Delta$ denote the standard gradient and Laplacian, while the positive constants $Ma$, $Ga$, and $Ca$ represent the Marangoni, Galileo, and capillary numbers, respectively. Referred to hereafter as the Davis equation, this model successfully predicts complex interfacial phenomena, including the formation and coarsening of dry spots \cite{vanhook, vanhook2}, which are critical failure modes in industrial coating applications.

The equation $(\ref{laiyuan})$ can be written in a more general form
\begin{align}\label{tuozhan}
	h_t=-\nabla\cdot\left( \theta(h)\nabla\Delta h+\psi(h)\nabla h\right).
\end{align}
According to \cite{cankao5,hou12,hou14}, the term $\theta(h)$ in Eq.$(\ref{tuozhan})$ represents surface tension effects; the term $\psi(h)$ reflects the
additional forces (gravity, thermocapillary effects or Van der Waals interactions for examples) and $h$ is the thickness of the film.

In one spatial dimension, the pioneering work of Bernis and Friedman \cite{cankao4} established the existence theory of nonnegative weak solutions for the purely capillary case ($\psi(h) \equiv 0$) with $\theta(h)=h^n$. Subsequent studies extensively explored the qualitative behavior of these solutions \cite{cankao3}, self-similar source-type solutions \cite{cankao5}, and their finite speed of propagation \cite{hou5}. Building upon this foundation, Bertozzi and Pugh \cite{cankao2} incorporated a nonnegative lower-order term ($\psi(h) \ge 0$), proving the existence of nonnegative weak solutions under specific assumptions on the ratio $\psi(h)/\theta(h)$.

In higher space dimensions and when the case $\psi(h)\equiv 0$, Dal Passo et al. \cite{hou7} prove the existence and positivity for $(\ref{tuozhan})$ with nonnegative initial values under appropriate boundary conditions using energy and entropy estimates. Gr\"un \cite{cankao6,cankao7} showed an existence result and a finite speed of propagation for model $(\ref{tuozhan})$ by using the entropy estimate and a generalization of Bernis' interpolation inequalities type. In the case $\theta(h)=h^n$ and $\psi(h)=h^m$, the existence of solutions and the finite speed of propagation property with nonlinear diffusion are developed in \cite{hou8} according to exponents $n$ and $m$. Recently, Shishkov and Taranets \cite{cankao1} investigated $(\ref{tuozhan})$ for $\theta(h)=|h|^n$ and $\psi(h)=|h|^m$; they constructed a nonnegative global generalized solution for nonnegative initial data in unbounded space and proved the
finiteness of the speed of propagation of the support of this solution.

Despite its physical relevance, an analytical framework for this fourth-order degenerate parabolic equation remains incomplete. The primary difficulty arises from the degenerate mobility, which causes a loss of uniform parabolicity as the film thickness $h$ approaches zero. This regime corresponds to impending rupture and makes questions of global existence and regularity non-trivial. While previous work \cite{wu} established global weak solutions for the one-dimensional Davis equation \eqref{laiyuan} under natural and flux-free boundary conditions, imposing non-homogeneous boundary conditions introduces new challenges in controlling boundary integrals within energy and entropy functionals.

Motivated by these mathematical intricacies, the present paper investigates the associated one-dimensional initial-boundary value problem. To formulate this problem, we let $\Omega=(a,d)$ be a bounded open interval and denote $Q_{T}:=(0,T)\times\Omega$ for a given final time $T>0$. Simplifying the notation, we replace the physical parameters $\frac{1}{3Ca}$, $\frac{Ga}{3}$, and $\frac{Ma}{2}$ with positive constants $\alpha_{1}$, $\alpha_{2}$, and $\alpha_{3}$ respectively. The Davis equation \eqref{laiyuan} then takes the following one-dimensional form
\begin{align}\label{1}
	h_t = -\left(\alpha_{1} h^3 h_{xxx} -\alpha_{2} h^3 h_{x} +\alpha_{3} h^2 h_{x}\right)_{x}, \qquad \text{in} \,\, Q_{T}.
\end{align}
We supplement this equation with the following initial and boundary conditions
\begin{align} \label{2}
	h&=c, \qquad {\rm on} \,\, [0,T]\times\partial\Omega,
	\\ \label{2.b}
	h_{xx}&=0, \qquad {\rm on} \,\, [0,T]\times\partial\Omega,
	\\ \label{3}
	h(0,x)&= h_{0}(x), \qquad x \in \Omega,
\end{align}
where $c$ is a positive constant. Fixing both the height and curvature corresponds to a physical setup where the contact line is pinned and the meniscus curvature is maintained. Mathematically, setting $h=c>0$ at the boundaries circumvents the issue of boundary degeneracy. Under this setup, our objective is to prove the global-in-time existence and non-negativity of weak solutions. By establishing these analytical properties, our work provides a mathematical foundation for the Davis model and bridges the gap between formal asymptotic derivations and numerical observations in thin-film hydrodynamics.
\\
\noindent\textbf{Statement of the main result.} We begin by introducing the essential function spaces and notations. Standard conventions are employed for Lebesgue and Sobolev spaces. We define $H_c^1(\Omega)$ as the set of functions in $H^1(\Omega)$ whose trace on the boundary $\partial\Omega$ equals the constant $c$. The dual spaces of $H^1_0(\Omega)$ and $H^2_0(\Omega)$ are denoted by $H^{-1}(\Omega)$ and $H^{-2}(\Omega)$ respectively. Furthermore, for any measurable spatial or space-time domain $Z$, including $\Omega$ and $Q_T$, the standard $L^2(Z)$ inner product is defined as
$$
(v_1,v_2)_Z = \int_{Z} v_1 v_2,
$$
where the integration is with respect to the Lebesgue measure.

We are now in a position to rigorously define the weak solutions to the initial-boundary value problem $(\ref{1})-(\ref{3})$ as follows.

\begin{definition}\label{dingyi1}
	Let $h_0\in H_c^1(\Omega)$. A function $h=h(t,x)$ with
	\begin{align}\label{d1}
		h \in L^2(0,T;H^2(\Omega))\cap L^{\infty}(0,T;H_c^1(\Omega)) 
	\end{align}
	is called a weak solution to the problem $(\ref{1})-(\ref{3})$, if
	\begin{align}\label{d2}
		&(h,\varphi_t)_{Q_{T}}-\alpha_{1} ( h_{xx}, (h^3\varphi_{x})_{x})_{Q_{T}} 
		\nonumber\\
		=&\alpha_{2} (h^3 h_{x}, \varphi_{x})_{Q_{T}} - \alpha_{3} (h^2 h_{x}, \varphi_{x})_{Q_{T}}-(h_0,\varphi(0))_{\Omega}
	\end{align}
	holds for any $\varphi \in C^\infty([0,T] \times \bar{\Omega})$ satisfying $\varphi = 0$ on $[0,T] \times \partial\Omega$ and $\varphi(T, \cdot) = 0$.
\end{definition}

The main result of this article can be stated as follows.
\begin{theorem}\label{jieguo}
	Suppose that $h_0\in H_c^{1}(\Omega)$ satisfies
	\begin{align}\label{jiashe}
		\frac{1}{h_0}\in L^1(\Omega) \quad \text{and} \quad h_0 \geq 0 \ \ a. e. \ \text{in} \ \Omega.
	\end{align}
	Then for any $T>0$, there exists a weak solution $h$ to the problem $(\ref{1})-(\ref{3})$ in the sense of Definition $\ref{dingyi1}$. This solution satisfies
	\begin{align}\label{2.9}
		&h_t\in L^{2}(0,T;H^{-1}(\Omega)),
		\\\label{daoshujieguo}
		&\frac{1}{h}\in L^{\infty}(0,T;L^1(\Omega)),
		\\\label{jieguozheng}
		&h \geq 0 \ \ \text{in} \ Q_T.
    \end{align}
	Moreover, the set $\{x \in \Omega \mid h(t,x)=0\}$ is of measure zero for any $t\in [0,T]$.
\end{theorem}

\begin{rmk}
	The nonnegativity $h \ge 0$ is mathematically guaranteed by the strong degeneracy of $h^3$ as $h \to 0$, which naturally prevents the solution from crossing the zero level.
\end{rmk}
\begin{rmk}
	The boundary condition $h=c>0$ prevents degeneracy at the boundary, avoiding the common difficulties associated with $h=0$. Furthermore, the regularities $h \in L^2(0,T;H^2(\Omega))$ and $1/h \in L^{\infty}(0,T;L^1(\Omega))$ are strictly consistent with this positive boundary assumption.
\end{rmk}

We outline the primary mathematical difficulties in the proof of Theorem $\ref{jieguo}$ and the strategies developed to overcome them.

{\em First}, the principal part of the fourth-order parabolic equation contains a mobility term $\alpha_1h^3$ that is not uniformly bounded from below. Consequently, a priori estimates of the highest derivatives contain a weight of the unknown $h$, preventing the direct application of standard compactness lemmas. To address this, we introduce a smooth regularization of the degenerate term to obtain a uniformly parabolic equation, which is then solved using the Galerkin method.

{\em Second}, fourth-order operators lack a classical maximum principle. The non-negativity of the solution must instead be established through specific entropy-functional inequalities tailored to the nonlinear structure of the equation.

{\em Third}, extending the framework of \cite{wu} to our boundary conditions introduces new challenges in controlling boundary integrals. Our strategies are summarized as follows:
\begin{itemize}
	\item [(1)] \textbf{Reconstruction of Energy Estimates:} The functional used in \cite{wu} cannot establish the uniform boundedness of $h$ in $L^\infty(0,T;H^1(\Omega))$ due to boundary terms arising from our specific boundary conditions. We construct a modified energy functional that controls these boundary contributions and yields the required estimates.
	
	\item [(2)] \textbf{Boundary-Adjusted Entropy Estimates:} To handle the boundary contributions in the entropy estimation, we introduce the modified entropy functions (\ref{6.2}). This neutralizes the boundary effects and closes the energy-entropy inequality.
	
	\item [(3)] \textbf{Boundary Integration:} The conditions $h=c$ and $h_{xx}=0$ require a precise treatment of boundary integrals during integration by parts. For instance, the boundary term in (\ref{2.6.11}) must be bounded appropriately to maintain the stability of the weak formulation.
\end{itemize}

\section{Existence for the approximate problem}
In this section, we consider an approximate version of the original initial-boundary value problem $(\ref{1})-(\ref{3})$ and prove the existence of its global weak solution.

The approximate problem is constructed as follows
\begin{align}\label{a1}
	h_t + \left(\alpha_{1} |h|_{\kappa}^3 h_{xxx} -\alpha_{2} |h|_{\kappa}^3 h_{x} +\alpha_{3} |h|_{\kappa}^2 h_{x}\right)_{x}&=0, \qquad \qquad {\rm in} \,\, Q_{T},
\end{align}
with the initial data and boundary conditions
\begin{align}\label{a2}
	h&=c, \qquad {\rm on} \,\, [0,T]\times\partial\Omega,
	\\ \label{a2b}
	h_{xx}&=0, \qquad {\rm on} \,\, [0,T]\times\partial\Omega,
	\\ \label{a3}
	h(0,x)&= h_{0}(x), \qquad x \in \Omega.
\end{align}
Here 
\begin{align}\label{kappa}
	|h|_{\kappa}:=\sqrt{|h|^2+\kappa^2}
\end{align}
and $\kappa\in(0,1]$. This definition of $|h|_{\kappa}$ is reasonable because we will consider the limit as $\kappa\rightarrow 0$, it helps to regularize the degenerate problem $(\ref{1})-(\ref{3})$.

For the initial-boundary value problem $(\ref{a1})-(\ref{a3})$ we introduce a new definition of the weak solutions as follows.

\begin{definition}\label{def2}
	Let $h_0\in H_c^1(\Omega)$. A function $h=h(t,x)$ with $(\ref{d1})$
	is a weak solution of the problem $(\ref{a1})-(\ref{a3})$, if
	\begin{align}\label{ruojie2}
		&(h,\varphi_t)_{Q_{T}}-\alpha_{1} ( h_{xx}, (|h|_{\kappa}^3\varphi_{x})_{x})_{Q_{T}} 
		\nonumber\\
		=&\alpha_{2} (|h|_{\kappa}^3 h_{x}, \varphi_{x})_{Q_{T}} - \alpha_{3} (|h|_{\kappa}^2 h_{x}, \varphi_{x})_{Q_{T}}-(h_0,\varphi(0))_{\Omega}
	\end{align}
	holds for any $\varphi \in C^\infty([0,T] \times \bar{\Omega})$ satisfying $\varphi = 0$ on $[0,T] \times \partial\Omega$ and $\varphi(T, \cdot) = 0$.
\end{definition}

Then we are going to construct approximate solutions to problem $(\ref{a1})-(\ref{a3})$. We first choose the bases functions $\{\omega_{i}\}^{m}_{i=1}$ with $m \in \mathbb{N_+}$, such that
\begin{align*}
	-\omega_{ixx} &= \lambda_{i} \omega_{i}, \qquad {\rm in} \,\, \Omega,
	\\
	\omega_{i} &= 0, \qquad\quad \, {\rm on} \,\, \partial\Omega,
\end{align*}
and $(\omega_{i},\omega_{j}) = \delta_{ij}$. Here $\delta_{ij}$ is the Kronecker symbol. Let $\{h^{m}\}_{m\in \mathbb{N_+}}$ be the approximate solutions defined by 
\begin{align*}
	h^m(t) - c = \sum_{i=1}^{m}g_{im}\omega_{i}\in C^\infty,
\end{align*}
where $g_{im}\in{\mathbb{R}}$ are functions of $t$ to be determined. The initial value $h_0$ is approximated by
\begin{align}\label{chuzhi2}
	h_0^m - c = \sum_{i=1}^{m}g_{im}(0)\omega_{i}\in C^\infty,
\end{align}
such that $\left\|h_0^m-h_0\right\|_{H^1(\Omega)}\rightarrow 0$ as $m \rightarrow \infty$. Thus the coefficients $g_{im}$ are determined by the following system of ordinary differential equations with initial condition $(\ref{chuzhi2})$:
\begin{align*}
	(h_t^m, \omega_{j}) - (\alpha_{1} |h^m|_{\kappa}^3 h_{xxx}^m -\alpha_{2} |h^m|_{\kappa}^3 h_{x}^m +\alpha_{3} |h^m|_{\kappa}^2 h_{x}^m, \omega_{jx}) = 0
\end{align*}
for $1 \leq j \leq m$. Thus, determining the approximate solutions reduces to solving the following Cauchy problem
\begin{align}\label{changweifen}
	\frac{d}{dt}g_{jm} &= F_{j}(g_{1m},\cdots,g_{mm},t),
	\nonumber\\
	g_{jm}(0) &= (h_0-c, \omega_{j}),
\end{align}
where
\begin{align*}
	F_{j}(g_{1m},\cdots,g_{mm},t) = &-\alpha_1\sum_{k=1}^{m}\lambda_k g_{km}\int_{\Omega}\left|\sum_{i=1}^{m}g_{im}\omega_i+c\right|_{\kappa}^3 \omega_{kx}\omega_{jx} dx
	\nonumber\\
	&-\alpha_2\sum_{l=1}^{m}g_{lm}\int_{\Omega}\left|\sum_{i=1}^{m}g_{im}\omega_i+c\right|_{\kappa}^3 \omega_{lx}\omega_{jx} dx 
	\nonumber\\
	&+\alpha_3\sum_{l=1}^{m}g_{lm}\int_{\Omega}\left|\sum_{i=1}^{m}g_{im}\omega_i+c\right|_{\kappa}^2 \omega_{lx}\omega_{jx} dx.
\end{align*}
This is a system of ODEs, therefore, by the standard theory of existence of solutions to an ODE system, we conclude that there exists a solution $g_{im}$ on $[0, t_m]$ satisfying problem $(\ref{changweifen})$.

To prove the global existence of solutions to problem $(\ref{a1})-(\ref{a3})$, one needs to obtain \textit{a priori} estimates that are uniform with respect to $m$ and valid for all $t \in [0,T]$, as established in the following lemmas.

\begin{lema}\label{jinsixianyan1}
	There exists a constant $C_\kappa$ independent of $m$, such that for any $t \in [0,T]$, the following estimate hold
	\begin{align}
		\left\| h^m \right\|_{L^\infty (0,t; H^1(\Omega))} \leq C_\kappa,
		\\
		\left\| |h^m|_{\kappa}^3 h_{xxx}^m \right\|_{L^{2}(Q_t)} \leq C_\kappa,
		\\
		\left\| h_{xx}^m \right\|_{L^2(Q_t)} \leq C_\kappa.
	\end{align}
\end{lema}

\begin{lema}\label{jinsixianyan2}
	There exists a constant $C_\kappa$ independent of $m$, such that for any $t \in [0,T]$, the following estimate holds
	\begin{align}
		\left\| h_t^m \right\|_{L^{2}(0,t; H^{-1}(\Omega))} \leq C_\kappa.
	\end{align}
\end{lema}

\noindent For the sake of brevity, we omit the proofs of Lemma \ref{jinsixianyan1} and Lemma \ref{jinsixianyan2}, as they are analogous to the corresponding results in Section 3.

\begin{lema}[Aubin-Lions]\label{aubin}
	Let $B_0$, $B$ and $B_1$ be Banach spaces such that $B_0$ and $B_1$ are reflexive, $B_0$ is compactly embedded into $B$, and $B$ is embedded into $B_1$. For $1\leq p_0, p_1 \leq +\infty$, define the space:
	$$
	W=\left\{f \ \Big| \ f\in L^{p_0}(0,T;B_0), \frac{df}{dt}\in L^{p_1}(0,T;B_1)\right\}.
	$$
	$\mathrm{(1)}$ If $p_0<+\infty$, then the embedding of $W$ into $L^{p_0}(0,T;B)$ is compact.
	\\
	$\mathrm{(2)}$ If $p_0=+\infty$ and $p_1>1$, then the embedding of $W$ into $C([0,T];B)$ is compact.
\end{lema}
\noindent For a proof of this lemma, we refer to, e.g., \cite{aubin-lions1,aubin-lions2,aubin-lions3}.

By applying Lemma $\ref{jinsixianyan1}-\ref{aubin}$, we will establish the following main theorem in this section, which concerns the existence of global solutions to the approximate problem $(\ref{a1})-(\ref{a3})$.

\begin{theorem}[Global Existence]\label{globalexistence}
	Assume that $h_0$ satisfies assumption $(\ref{jiashe})$. Then there exists a global weak solution $h$ in the sense of Definition \ref{def2} to problem $(\ref{a1})-(\ref{a3})$ such that, for any given $\kappa$,
	\begin{align}\label{kappajie1}
		\left\| h \right\|_{L^\infty (0,T; H^1(\Omega))} + \left\| h \right\|_{L^2 (0,T; H^3(\Omega))} \leq C_\kappa,
		\\\label{kappajie2}
		\left\| h_t \right\|_{L^{2}(0,T; H^{-1}(\Omega))} \leq C_\kappa.
	\end{align}
\end{theorem}
\begin{proof}
Based on Lemmas $\ref{jinsixianyan1}$ and $\ref{jinsixianyan2}$, we have
\begin{align}\label{yizhixing}
	\left\| h^m \right\|_{L^\infty (0,t; H^1(\Omega))} + \left\| h^m \right\|_{L^2 (0,t; H^3(\Omega))} + \left\| h_t^m \right\|_{L^{2}(0,t; H^{-1}(\Omega))} \leq C_\kappa,
\end{align}
where the constant $C_\kappa$ is independent of $m$. Because these estimates are independent of time, they guarantee that the local approximate solutions $h^m$ do not blow up in finite time, thereby allowing us to extend the solutions globally to the entire interval $[0, T]$. 

\noindent We now choose a sequence $h_0^m \in C^\infty(\bar{\Omega})$ such that
\begin{align}\label{000}
	\left\| h_0^m-h_0 \right\|_{H^1(\Omega)}\rightarrow 0.
\end{align}
To apply the Aubin-Lions lemma, we choose $p_0=2$, $p_1=2$, and
\begin{align*}
	B_0=H^3(\Omega), \quad B=C^{2+\alpha}(\bar{\Omega}), \quad B_1=H^{-1}(\Omega)
\end{align*}
for a suitable positive constant $\alpha$. It is easy to verify that $B_0$, $B$, $B_1$ satisfy the conditions of Lemma $\ref{aubin}$. Thus, $(\ref{yizhixing})$ implies that there exists a subsequence, which we still denote by $h^m$, and a limit function $h$, such that as $m\rightarrow \infty$,
\begin{align*}
	\left\| h^m-h \right\|_{L^2(0,t; C^{2+\alpha}(\bar{\Omega}))}\rightarrow 0.
\end{align*}
Furthermore, by choosing $B=H^2(\Omega)$, we obtain
\begin{align}\label{shoulianH2}
	\left\| h^m-h \right\|_{L^2(0,t; H^2(\Omega))}\rightarrow 0.
\end{align}
Next, choosing $p_0=\infty, p_1=2$, and
\begin{align*}
	B_0=H^1(\Omega), \quad B=C^{\alpha}(\bar{\Omega}), \quad B_1=H^{-1}(\Omega),
\end{align*}
we conclude that
\begin{align}\label{shoulianwuqiong}
	\left\| h^m - h \right\|_{C([0,t];C^{\alpha}(\bar{\Omega}))}\rightarrow 0.
\end{align}
It follows from $(\ref{yizhixing})$ that we can extract a subsequence satisfying
\begin{align}\label{h3x}
	h^m\overset{*}{\rightharpoonup} h \quad {\rm in}\quad L^\infty(0,t; H^1(\Omega)), \quad h_{xxx}^m\rightharpoonup h_{xxx} \quad {\rm in} \quad L^2(Q_t),
\end{align}
and the limit function $h$ naturally inherits the bounds
\begin{align*}
	&h\in L^\infty(0,t; H^1(\Omega)) \cap L^2(0,t; H^3(\Omega)),
	\\
	&h_t \in L^{2}(0,t;H^{-1}(\Omega)),
\end{align*}
which yield $(\ref{kappajie1})$ and $(\ref{kappajie2})$. 

\noindent Next, we derive convergence estimates for the nonlinear terms to show that $h$ is a solution of $(\ref{a1})$. We observe that
\begin{align*}
	&\left| |h^m|_{\kappa}^3 - |h|_{\kappa}^3 \right| 
	\nonumber\\
	&= \left| (\sqrt{|h^m|^2+\kappa^2})^3-(\sqrt{|h|^2+\kappa^2})^3 \right|
	\nonumber\\
	&= \left| \sqrt{|h^m|^2+\kappa^2}-\sqrt{|h|^2+\kappa^2} \right| \left| |h^m|^2+\kappa^2 +\sqrt{|h^m|^2+\kappa^2}\sqrt{|h|^2+\kappa^2} + |h|^2+\kappa^2 \right|
	\nonumber\\
	&\leq \left| h^m-h \right| \left| |h^m|^2+\kappa^2 +\sqrt{|h^m|^2+\kappa^2}\sqrt{|h|^2+\kappa^2} + |h|^2+\kappa^2 \right|
\end{align*}
and
\begin{align*}
	&\left| |h^m|_{\kappa}^2 - |h|_{\kappa}^2 \right|
	\nonumber\\
	&= \left| |h^m|^2-|h|^2 \right|
	\nonumber\\
	&\leq \left| h^m-h \right| \left| |h^m|+|h| \right|.
\end{align*}
Combined with the uniform convergence from $(\ref{shoulianwuqiong})$, we deduce
\begin{align}\label{shoulianL2}
	&\left\| |h^m|_{\kappa}^3-|h|_{\kappa}^3 \right\|_{L^{\infty}(Q_t)}
	\nonumber\\
	&\leq \left\| h^m-h \right\|_{L^{\infty}(Q_t)} \left|\left| |h^m|^2+\kappa^2 +\sqrt{|h^m|^2+\kappa^2}\sqrt{|h|^2+\kappa^2} + |h|^2+\kappa^2 \right|\right|_{L^{\infty}(Q_t)}\rightarrow 0
\end{align}
and
\begin{align}\label{shoulianL3}
	&\left\| |h^m|_{\kappa}^2-|h|_{\kappa}^2 \right\|_{L^{\infty}(Q_t)}
	\nonumber\\
	&\leq \left\| h^m-h \right\|_{L^{\infty}(Q_t)} \left|\left| |h^m|+|h| \right|\right|_{L^{\infty}(Q_t)}\rightarrow 0.
\end{align}
Combining $(\ref{h3x})$ with $(\ref{shoulianL2})$, one concludes that
\begin{align}\label{1111}
	|h^m|_{\kappa}^3 h_{xxx}^m \rightharpoonup |h|_{\kappa}^3 h_{xxx} \qquad {\rm weakly} \ {\rm in} \ L^2(Q_t).
\end{align}
Then, from $(\ref{shoulianH2})$ and $(\ref{shoulianL2})$, it follows that
\begin{align}\label{222}
	| h^m |_{\kappa}^3 h_{x}^m\rightarrow |h|_{\kappa}^3 h_{x} \qquad {\rm strongly} \ {\rm in} \ L^2(Q_t).
\end{align}
Finally, $(\ref{shoulianH2})$ and $(\ref{shoulianL3})$ imply that 
\begin{align}\label{333}
	| h^m |_{\kappa}^2 h_{x}^m\rightarrow |h|_{\kappa}^2 h_{x} \qquad {\rm strongly} \ {\rm in} \ L^{2}(Q_t).
\end{align}
Using $(\ref{000})$, $(\ref{1111})$, $(\ref{222})$ and $(\ref{333})$ to pass to limit in the Galerkin formulation, we arrived at
\begin{align}\label{jinsiruojie}
	&(h,\varphi_t)_{Q_{T}}+\alpha_{1} ( |h|_{\kappa}^3h_{xxx}, \varphi_{x})_{Q_{T}} 
	\nonumber\\
	=&\alpha_{2} (|h|_{\kappa}^3 h_{x}, \varphi_{x})_{Q_{T}} - \alpha_{3} (|h|_{\kappa}^2 h_{x}, \varphi_{x})_{Q_{T}}-(h_0,\varphi(0))_{\Omega}.
\end{align}
Finally, we integrate the second term on the left-hand side of $(\ref{jinsiruojie})$ by parts with respect to $x$. Because the boundary condition $(\ref{a2b})$ ensures $h_{xx} = 0$ on $\partial\Omega$, the boundary terms vanish, yielding
\begin{align*}
	&(h,\varphi_t)_{Q_{T}}-\alpha_{1} ( h_{xx}, (|h|_{\kappa}^3\varphi_{x})_{x})_{Q_{T}} 
	\nonumber\\
	=&\alpha_{2} (|h|_{\kappa}^3 h_{x}, \varphi_{x})_{Q_{T}} - \alpha_{3} (|h|_{\kappa}^2 h_{x}, \varphi_{x})_{Q_{T}}-(h_0,\varphi(0))_{\Omega}.
\end{align*}
This implies that $(\ref{ruojie2})$ holds. The proof of Theorem $\ref{globalexistence}$ is thus complete.
\end{proof}

\section{A priori estimates independent of $\kappa$}
In Section 2, the global existence of weak solutions depending on $\kappa$ was established. In this section, we derive \textit{a priori} estimates for problem $(\ref{a1})-(\ref{a3})$ that are uniform with respect to $\kappa\in (0,1]$ for any fixed $T>0$.
\begin{lema}
	There exists a constant $C$, independent of $\kappa$, such that for any $T<\infty$ and $t\in [0,T]$, the following estimates hold
	\begin{align}\label{4.3}
	    \left\|	h^{\kappa} \right\|_{L^\infty (0,t; H^1(\Omega))}&\leq C,
		\\\label{4.4}
		\left\| h^{\kappa}\right\|_{L^\infty(Q_{t})}&\leq C,
		\\\label{4.5}
		\left\| |h^{\kappa}|_{\kappa}^\frac{3}{2}|h^{\kappa}_{xxx}\right\|_{L^2(Q_{t})}&\leq C,
		\\\label{4.66}
		\left\| |h^{\kappa}|_{\kappa}^3|h^{\kappa}_{xxx}\right\|_{L^2(Q_{t})}&\leq C.
	\end{align}
\end{lema}
\begin{proof}
It is evident that $(\ref{4.4})$ follows directly from $(\ref{4.3})$ by the Sobolev embedding theorem. Furthermore, $(\ref{4.66})$ can be obtained by combining $(\ref{4.4})$ and $(\ref{4.5})$ with H\"older's inequality. Thus, we focus on proving $(\ref{4.3})$ and $(\ref{4.5})$. To this end, we define the following energy functional
\begin{align}
	\begin{split}\label{ziyounenggai}
		F_{\kappa,c}[h^{\kappa}]=&\int_{\Omega}\frac{\alpha_1}{2}|h^{\kappa}_{x}|^2+\Psi_{\kappa,c}(h^{\kappa})dx,
		\\
		\Psi_{\kappa,c}(h^{\kappa})=&\frac{\alpha_2}{2}|h^{\kappa}|^2-\alpha_{3}h^{\kappa}\ln\left(h^{\kappa}+|h^{\kappa}|_{\kappa}\right)
		\\
		&+\alpha_{3}|h^{\kappa}|_{\kappa}+\left(\alpha_{3}\ln\left(c+|c|_{\kappa}\right)-\alpha_{2}c\right)h^{\kappa}+\bar{C},
	\end{split}
\end{align}
where $\bar{C}$ is a sufficiently large positive constant chosen such that
\begin{align}\label{4.7}
	\Psi_{\kappa,c}(y)\geq\frac{\alpha_2}{4}y^2, \qquad \forall y\in\mathbb{R}, \kappa\in(0,1].
\end{align}
Direct calculation shows that the second derivative of $\Psi_{\kappa,c}$ is $\Psi_{\kappa,c}''(h) = \alpha_2 - \alpha_3 |h|_\kappa^{-1}$.
Using $(\ref{a1})-(\ref{a2b})$, and integrating by parts twice, we obtain
\begin{align}\label{22.22}
	\frac{d}{dt} F_{\kappa,c}[h^\kappa] &=\int_{\Omega}\alpha_1 h^{\kappa}_{x} h^{\kappa}_{xt} +\Psi_{\kappa,c}'(h^{\kappa})h^{\kappa}_t dx 
	\nonumber\\
	&=\int_{\Omega}(-\alpha_1 h^{\kappa}_{xx}+\Psi_{\kappa,c}'(h^{\kappa}))h^{\kappa}_t dx
	\nonumber\\
	&=\int_{\Omega}(\alpha_1 h^{\kappa}_{xx}-\Psi_{\kappa,c}'(h^{\kappa})) \left(\alpha_{1} |h^{\kappa}|_{\kappa}^3 h^{\kappa}_{xxx} -\alpha_{2} |h^{\kappa}|_{\kappa}^3 h^{\kappa}_{x} +\alpha_{3} |h^{\kappa}|_{\kappa}^2 h^{\kappa}_{x}\right)_{x} dx
	\nonumber\\
	&=-\int_{\Omega}(\alpha_1 h^{\kappa}_{xxx}-\Psi_{\kappa,c}''(h^{\kappa}) h^{\kappa}_{x}) \left(\alpha_{1} |h^{\kappa}|_{\kappa}^3 h^{\kappa}_{xxx} -\alpha_{2} |h^{\kappa}|_{\kappa}^3 h^{\kappa}_x +\alpha_{3} |h^{\kappa}|_{\kappa}^2 h^{\kappa}_{x}\right) dx
	\nonumber\\
	&=-\int_{\Omega}\frac{1}{|h^{\kappa}|_{\kappa}^3} \left(\alpha_{1} |h^{\kappa}|_{\kappa}^3 h^{\kappa}_{xxx} -\alpha_{2} |h^{\kappa}|_{\kappa}^3 h^{\kappa}_{x} +\alpha_{3} |h^{\kappa}|_{\kappa}^2 h^{\kappa}_{x}\right)^2 dx \leq 0.
\end{align}
Integrating $(\ref{22.22})$ over $[0, t]$ yields
\begin{align}\label{4.9}
	F_{\kappa,c}[h^{\kappa}]+\int_{0}^{t}\int_{\Omega}\left(\alpha_{1} |h^{\kappa}|_{\kappa}^{\frac{3}{2}} h^{\kappa}_{xxx} -\alpha_{2} |h^{\kappa}|_{\kappa}^{\frac{3}{2}} h^{\kappa}_{x} +\alpha_{3} |h^{\kappa}|_{\kappa}^{\frac{1}{2}}  h^{\kappa}_{x}\right)^2 dxd\tau=F_{\kappa,c}[h_0].
\end{align}
It follows from the condition $h_0\in H^1(\Omega)$, $(\ref{ziyounenggai})$ and the Sobolev embedding theorem that there exists a constant $C$ such that
\begin{align}\label{4.10}
	F_{\kappa,c}[h^{\kappa}]\leq F_{\kappa,c}[h_0] \leq C.
\end{align}
Together with $(\ref{4.7})$, one arrives at $(\ref{4.3})$.

\noindent In what follows, we give the proof of $(\ref{4.5})$. For this aim, we introduce a useful inequality. There exist positive constants $\lambda, \mu, \nu$ such that
\begin{equation}\label{4.12}
	\lambda a^2 \leq (a+b+c)^2+\mu b^2+\nu c^2, \qquad \forall \ a,b,c\in\mathbb{R}.
\end{equation}
From $(\ref{ziyounenggai})-(\ref{4.10})$, we obtain
\begin{equation}\label{4.11}
	\int_{0}^{t}\int_{\Omega}\left(\alpha_{1} |h^{\kappa}|_{\kappa}^{\frac{3}{2}} h^{\kappa}_{xxx} -\alpha_{2} |h^{\kappa}|_{\kappa}^{\frac{3}{2}} h^{\kappa}_{x} +\alpha_{3} |h^{\kappa}|_{\kappa}^{\frac{1}{2}}  h^{\kappa}_{x}\right)^2 dxd\tau\leq C.
\end{equation}
By $(\ref{4.3})$ and $(\ref{4.4})$, the lower-order terms satisfy
\begin{equation}\label{4.13}
	\int_{0}^{t}\int_{\Omega} \left( \alpha_{2} |h^{\kappa}|_{\kappa}^{\frac{3}{2}} h^{\kappa}_{x} \right)^2 dxd\tau \leq C \| \, |h^{\kappa}|+\kappa \, \|_{L^\infty(Q_{t})}^3 \int_{0}^{t}\int_{\Omega} |h^{\kappa}_x|^2 dxd\tau \leq C
\end{equation}
and
\begin{equation}\label{4.14}
	\int_{0}^{t}\int_{\Omega} \left( \alpha_{3} |h^{\kappa}|_{\kappa}^{\frac{1}{2}} h^{\kappa}_{x} \right)^2 dxd\tau \leq C \| \, |h^{\kappa}|+\kappa \, \|_{L^\infty(Q_{t})} \int_{0}^{t}\int_{\Omega} |h^{\kappa}_x|^2 dxd\tau \leq C.
\end{equation}
Applying inequality $(\ref{4.12})$ to the estimates $(\ref{4.11})-(\ref{4.14})$, we obtain $(\ref{4.5})$. This completes the proof of this lemma.
\end{proof}

To obtain higher-order spatial regularity, we construct a Bernis-Friedman type entropy functional. This allows us to establish a uniform estimate for the second-order spatial derivative.
\begin{lema}\label{lm4.2}
	There exists a constant $C$ independent of $\kappa$, such that for any $t\in[0,T]$ and $h_0$ satisfying $(\ref{jiashe})$, there holds
	\begin{align}\label{D2}
		\left\| h_{xx}^{\kappa} \right\|_{L^2(Q_t)} \leq C.
	\end{align}
\end{lema}
\begin{proof}
	We introduce the following entropy functions
	\begin{align}\label{6.2}
		G_\kappa(s)=\int_{c}^{s}g_\kappa(r) dr, \qquad
		g_\kappa(s)=\int_{c}^{s}\frac{1}{|r|_{\kappa}^3} dr,
	\end{align}
	where $c>0$ is the constant given in $(\ref{a2})$. It is easy to verify that
	\begin{align}\label{6.3b}
		&G_\kappa'(s)=g_\kappa(s),\quad  G_\kappa''(s)=g_\kappa'(s)=\frac{1}{|s|_{\kappa}^3},
		\\ \label{6.5}
		&0 \leq G_\kappa(s)\leq G_0(s), \quad \forall s\in \mathbb{R},
	\end{align}
	where $G_0=\lim\limits_{\kappa\rightarrow0}G_\kappa$. Integrating by parts, we calculate that
	\begin{align}\label{G}
		G_\kappa(s)=&\int_{s}^{c}\int_{t}^{c}\frac{1}{|r|_{\kappa}^3}drdt
		\nonumber\\
		=&
		\frac{s}{\kappa^2}\left( \frac{s}{\sqrt{s^2+\kappa^2}}-\frac{c}{\sqrt{c^2+\kappa^2}} \right) + \frac{1}{\sqrt{s^2+\kappa^2}}-\frac{1}{\sqrt{c^2+\kappa^2}}
		\nonumber\\
		=& \frac{s^3}{\sqrt{c^2+\kappa^2}\sqrt{s^2+\kappa^2}\left( \sqrt{c^2+\kappa^2}s+c\sqrt{s^2+\kappa^2} \right)}
		\nonumber\\
		&-\frac{c^2s}{\sqrt{c^2+\kappa^2}\sqrt{s^2+\kappa^2}\left( \sqrt{c^2+\kappa^2}s+c\sqrt{s^2+\kappa^2} \right)} 
		\nonumber\\
		&+ \frac{1}{\sqrt{s^2+\kappa^2}}-\frac{1}{\sqrt{c^2+\kappa^2}},
	\end{align}
	and taking the limit yields
	\begin{align}\label{g}
		G_0(s)=\left\{
		\begin{aligned}
			&\frac{s}{2c^2}+\frac{1}{2s}-\frac{1}{c}, \quad &s > 0,\\
			&\infty, \quad &s\leq 0.
		\end{aligned}\right.
	\end{align}
	Multiplying equation $(\ref{a1})$ by $g_\kappa(h^{\kappa})$, integrating over $Q_t$, and applying integration by parts, with $(\ref{a2})$ and $(\ref{6.3b})$, we obtain
	\begin{align}\label{6.7}
		&\int_{\Omega}G_\kappa(h^{\kappa})dx+\int_{0}^{t}\int_{\Omega}\alpha_{1}|h_{xx}^{\kappa}|^2+\alpha_{2}|h_{x}^{\kappa}|^2dxd\tau
		\nonumber\\
		=&\int_{\Omega}G_\kappa(h_0(x))dx+\int_{0}^{t}\int_{\Omega}\alpha_{3}\frac{1}{|h^{\kappa}|_\kappa}|h_{x}^{\kappa}|^2dxd\tau.
	\end{align}
	From $h_0\in H^1(\Omega)$, $(\ref{jiashe})$, $(\ref{6.5})$ and $(\ref{g})$, we have
	\begin{align}
		\int_{\Omega}G_\kappa(h_0(x))dx\leq \int_{\Omega}G_{0}(h_0(x))dx \leq C.
	\end{align}
	To control the last term in $(\ref{6.7})$, we integrate by parts and apply Young's inequality and $(\ref{a2})$, we obtain
    \begin{align}\label{2.6.11}
    	&\int_{0}^{t}\int_{\Omega}\alpha_{3}\frac{|h_x^{\kappa}|^2}{|h^{\kappa}|_{\kappa}}dxd\tau
    	\nonumber\\
    	=&\alpha_{3}\int_{0}^{t}\int_{\Omega}  \left( \ln\left(h^\kappa + |h^{\kappa}|_{\kappa}\right) \right)_{x} h^{\kappa}_x dxd\tau
    	\nonumber\\
    	=&-\alpha_{3}\int_{0}^{t}\int_{\Omega} \ln\left(h^\kappa + |h^{\kappa}|_{\kappa}\right) h^{\kappa}_{xx} dxd\tau +\alpha_{3} \ln\left(c + |c|_{\kappa}\right) \int_{0}^{t} \int_{\Omega} h^{\kappa}_{xx}dx d\tau
    	\nonumber\\
    	\leq& \int_{0}^{t}\int_{\Omega}\frac{\alpha_1}{2}|h^{\kappa}_{xx}|^2+C_{\alpha_1}\left(1+\ln^2\left(h^\kappa + |h^{\kappa}|_{\kappa}\right)\right) dxd\tau.
    \end{align}
	Combining $(\ref{6.7})-(\ref{2.6.11})$ yields
	\begin{align}\label{guocheng}
		\int_{\Omega}G_\kappa(h^{\kappa})dx + \int_{0}^{t}\int_{\Omega}\frac{\alpha_{1}}{2}|h_{xx}^{\kappa}|^2 dxd\tau \leq C\left(1+\int_{0}^{t}\int_{\Omega}\ln^2\left(h^\kappa + |h^{\kappa}|_{\kappa}\right)dxd\tau \right).
	\end{align}
	Notice that $G_\kappa(s)$ is coercive, and there exists a constant $C>0$ independent of $\kappa$ such that $\ln^2(s + \sqrt{s^2+\kappa^2}) \leq C(1 + G_\kappa(s))$ for all $s \in \mathbb{R}$. Together with $(\ref{guocheng})$, we yield
	\begin{align*}
		\int_{\Omega}G_\kappa(h^{\kappa})dx + \int_{0}^{t}\int_{\Omega}\frac{\alpha_1}{2}|h_{xx}^{\kappa}|^2dxd\tau \leq C\left(1+\int_{0}^{t}\int_{\Omega}G_\kappa(h^{\kappa})dxd\tau \right).
	\end{align*}
	Applying Gronwall's lemma, we arrive at
	\begin{align}\label{6.15}
		\int_{\Omega}G_\kappa(h^{\kappa})dx \leq C.
	\end{align}
	Consequently, returning to the energy equality yields
	\begin{align*}
		\int_{0}^{t}\int_{\Omega}|h^{\kappa}_{xx}|^2dxd\tau \leq C.
	\end{align*}
	This establishes $(\ref{D2})$, which completes the proof of this lemma.
\end{proof}

Finally, we derive uniform bounds for the time derivatives using duality arguments and the previously obtained spatial estimates.
\begin{lema}
	There exists a constant $C$, independent of $\kappa$, such that for any $t\in[0,T]$, the following estimates hold
	\begin{align}\label{lemma4.2}
		\left\|h^{\kappa}_t\right\|_{L^2(0,t;H^{-1}(\Omega))}\leq C,
		\\\label{4.22}
		\left\|h^{\kappa}_{xt}\right\|_{L^2(0,t;H^{-2}(\Omega))}\leq C.
	\end{align}
\end{lema}
\begin{proof}
	Utilizing $(\ref{4.3})$, $(\ref{4.4})$ and $(\ref{4.66})$, we integrate by parts and apply H\"older's inequality to obtain
	\begin{equation}
		\begin{split}\label{bei23}
			&|(h^{\kappa}_t, \varphi)_{Q_{t}}|
			\\
			=& \ \left|\left(\alpha_{1} |h^{\kappa}|_{\kappa}^3 h^{\kappa}_{xxx} -\alpha_{2} |h^{\kappa}|_{\kappa}^3 h^{\kappa}_x +\alpha_{3} |h^{\kappa}|_{\kappa}^2 h^{\kappa}_x,\varphi_x\right)_{Q_{t}}\right|
			\\
			\leq& \ \left|\left| \ \alpha_{1} |h^{\kappa}|_{\kappa}^3 h^{\kappa}_{xxx} -\alpha_{2} |h^{\kappa}|_{\kappa}^3 h^{\kappa}_x +\alpha_{3} |h^{\kappa}|_{\kappa}^2 h^{\kappa}_x \ \right|\right|_{L^2(Q_{t})}\left\|\varphi_x\right\|_{L^2(Q_{t})}
			\\
			\leq& \ \left(\alpha_{1}\left|\left| |h^{\kappa}|_{\kappa}^3 h^{\kappa}_{xxx}\right|\right|_{L^2(Q_{t})} + \alpha_{2}\left|\left| |h^{\kappa}|_{\kappa}^3 h^{\kappa}_x\right|\right|_{L^2(Q_{t})} +\alpha_{3}\left|\left| |h^{\kappa}|_{\kappa}^2 h^{\kappa}_x\right|\right|_{L^2(Q_{t})}\right)\left\|\varphi_x\right\|_{L^2(Q_{t})}
			\\
			\leq& \ C\left(1 + \left|\left| \, |h^{\kappa}| +\kappa \, \right|\right|_{L^\infty(Q_{t})}^3 \left\|h^{\kappa}_x\right\|_{L^2(Q_{t})} +\left|\left|\, |h^{\kappa}| +\kappa \,\right|\right|_{L^\infty(Q_{t})}^2 \left\|h^{\kappa}_x\right\|_{L^2(Q_{t})}\right) \left\|\varphi_x\right\|_{L^2(Q_{t})}
			\\
			\leq& \ C \left\|\varphi\right\|_{L^2(0,t;H_0^1(\Omega))}
		\end{split}
	\end{equation}
	for all test functions $\varphi \in L^2(0,t;H_0^1(\Omega))$. Similarly, we have 
	\begin{equation}
		\begin{split}\label{shijiandao2}
			&|(h^{\kappa}_{xt}, \psi)_{Q_{t}}|
			\\
			=& \ \left|\left(\alpha_{1} |h^{\kappa}|_{\kappa}^3 h^{\kappa}_{xxx} -\alpha_{2} |h^{\kappa}|_{\kappa}^3 h^{\kappa}_x +\alpha_{3} |h^{\kappa}|_{\kappa}^2 h^{\kappa}_x,\psi_{xx}\right)_{Q_{t}}\right|
			\\
			\leq& \ C\left(1 + \left|\left| \, |h^{\kappa}| +\kappa \, \right|\right|_{L^\infty(Q_{t})}^3 \left\|h^{\kappa}_x\right\|_{L^2(Q_{t})} +\left|\left|\, |h^{\kappa}| +\kappa \,\right|\right|_{L^\infty(Q_{t})}^2 \left\|h^{\kappa}_x\right\|_{L^2(Q_{t})}\right) \left\|\psi_{xx}\right\|_{L^2(Q_{t})}
			\\
			\leq& \ C \left\|\psi\right\|_{L^2(0,t;H_0^2(\Omega))}
		\end{split}
	\end{equation}
	for all test functions $\psi \in L^2(0,t;H_0^2(\Omega))$. Inequality $(\ref{bei23})$ directly proves $(\ref{lemma4.2})$, and $(\ref{shijiandao2})$ proves $(\ref{4.22})$. This completes the proof of the lemma.
\end{proof}

\section{Nonnegativity and existence of the original problem}
In this section, we utilize the \textit{a priori} estimates independent of $\kappa$, established in Section 3, to investigate the convergence of $h^\kappa$ as $\kappa \to 0$. We then prove the nonnegativity of the limit function $h$, thereby establishing the existence of solutions to the original problem $(\ref{1})-(\ref{3})$.

It follows from $(\ref{4.3})$ and $(\ref{D2})$ that
\begin{align*}
	\left\|h^{\kappa}\right\|_{L^2(0,T;H^2(\Omega))}\leq C.
\end{align*}
Hence, we can extract a subsequence $\{h^{\kappa_n}\}$, still denoted by $h^\kappa$, and find a function $h \in L^2(0,T;H^2(\Omega))$ such that
\begin{align}\label{fuzhu}
	h^{\kappa} \rightharpoonup h \quad \text{weakly in } L^2(0,T;H^2(\Omega)).
\end{align}

\begin{lema}\label{shoulianlema}
	Let $0<\alpha<\frac{1}{2}$. There exists a subsequence $\kappa_n\rightarrow0$ and a function $h\in C([0,T];C^{\alpha}(\overline{\Omega}))$ with
	\begin{align}\label{5.1}
		h&\in L^\infty(0,T;H^1(\Omega))\cap L^2(0,T;H^2(\Omega)),
		\\ \label{5.2}
		h_t&\in L^{2}(0,T;H^{-1}(\Omega)),
	\end{align}
	such that the sequence $h^{\kappa_n}$, still denoted by $h^{\kappa}$, satisfies
	\begin{align}\label{5.3}
		\left\|h^\kappa- h\right\|_{C([0,T];C^{\alpha}(\overline{\Omega}))}& \quad \rightarrow \quad 0,
		\\\label{hx}
		\left\|h_x^\kappa- h_x\right\|_{L^2(0,T;C^{\alpha}(\overline{\Omega}))}& \quad \rightarrow \quad 0,
		\\\label{5.4}
		\left|\left| \, |h^\kappa|^3h_x^\kappa - |h|^3h_x \, \right|\right|_{L^2(Q_T)}& \quad \rightarrow \quad 0,
		\\\label{5.5}
		\left|\left| \, |h^\kappa|^2h_x^\kappa - |h|^2h_x \, \right|\right|_{L^2(Q_T)}& \quad \rightarrow \quad 0,
		\\\label{weak2}
		|h^\kappa|_{\kappa}^3h_{xx}^\kappa \rightharpoonup |h|^3h_{xx}& \quad weakly \ \ in \ \ L^2(Q_T),
		\\\label{weak1}
		|h^\kappa|_{\kappa}h^{\kappa}h_{x}^{\kappa}h_{xx}^\kappa \rightharpoonup |h|h h_{x} h_{xx}& \quad weakly \ \ in \ \ L^1(Q_T).
	\end{align}
\end{lema}
\begin{proof}
	Applying Lemma $\ref{aubin}$, we choose $p_0=+\infty$, $p_1=2$ and
	\begin{equation*}
		B_0=H^1(\Omega),\quad B=C^\alpha(\overline{\Omega}),\quad B_1=H^{-1}(\Omega)
	\end{equation*}
	for any $0 < \alpha < \frac{1}{2}$. The estimates $(\ref{4.3})$ and $(\ref{lemma4.2})$ imply that the sequence $h^\kappa$ is uniformly bounded in $L^\infty(0,T; H^1(\Omega))$ and that $h^\kappa_t$ is uniformly bounded in $L^2(0,T; H^{-1}(\Omega))$. It follows from Lemma $\ref{aubin}$ and $(\ref{fuzhu})$ that there exists a subsequence, still denoted by $h^\kappa$, which converges strongly in $C([0,T];C^\alpha(\overline{\Omega}))$ to a limit function $h$. This proves $(\ref{5.3})$.

	\noindent Next, applying Lemma $\ref{aubin}$, we choose $p_0=2$, $p_1=2$ and 	
	\begin{equation*}
		B_0=H^1(\Omega),\quad B=C^\alpha(\overline{\Omega}),\quad B_1=H^{-2}(\Omega)
	\end{equation*}
	for any $0 < \alpha < \frac{1}{2}$. We deduce from $(\ref{4.3})$, $(\ref{D2})$ and $(\ref{4.22})$ that the sequence $h_x^\kappa$ is uniformly bounded in $L^2(0,T;H^1(\Omega))$ and $h^\kappa_{xt}$ is uniformly bounded in ${L^2(0,T;H^{-2}(\Omega))}$. Combining this with Lemma $\ref{aubin}$ and $(\ref{fuzhu})$ yields $(\ref{hx})$.

	\noindent Since the sequences $h^\kappa$ and $h^\kappa_t$ are uniformly bounded in $L^\infty(0,T;H^1(\Omega))\cap L^2(0,T;H^2(\Omega))$ and $L^2(0,T; H^{-1}(\Omega))$, respectively, we extract subsequences of $h^\kappa$ and $h^\kappa_t$ that converges weakly-$\ast$ to $h$ and $h_t$ in their respective spaces. Thus, we have $h\in L^\infty(0,T;H^1(\Omega))\cap L^2(0,T;H^2(\Omega))$ and $h_t\in L^{2}(0,T;H^{-1}(\Omega))$. This prove $(\ref{5.1})$ and $(\ref{5.2})$.
	
	\noindent By $(\ref{kappa})$, we have 
	\begin{align}\label{fuzhu2}
		\begin{split}
			\left|\left|\,|h^\kappa|_{\kappa}-|h|\,\right|\right|_{L^{\infty}(Q_{T})}&\leq\left|\left|\,|h^\kappa|_{\kappa}-|h^\kappa|\,\right|\right|_{L^{\infty}(Q_{T})} + \left|\left|\,|h^{\kappa}|-|h|\right|\right|_{L^{\infty}(Q_{T})}
			\\
			&\leq \kappa + \left|\left|\,h^{\kappa}-h\right|\right|_{L^{\infty}(Q_{T})}.
		\end{split}
	\end{align}
	
	\noindent From $(\ref{5.3})$ and $(\ref{fuzhu2})$, we deduce that
	$|h^\kappa|_{\kappa}$ converges uniformly to $|h|$ on $Q_{T}$. Consequently, $|h^\kappa|_{\kappa}^3$ and $|h^\kappa|_{\kappa}^2$ converge uniformly to $|h|^3$ and $|h|^2$ on $Q_{T}$, respectively. Combining this with $(\ref{hx})$ yields $(\ref{5.4})$ and $(\ref{5.5})$.
	
	\noindent By $(\ref{D2})$ and $(\ref{fuzhu})$, we obtain
	\begin{align}\label{hxx}
		h_{xx}^\kappa \rightharpoonup h_{xx} \quad \text{weakly in } L^2(Q_T).
	\end{align}
	Combining this with the uniform convergence of $|h^\kappa|_{\kappa}^3$ to $|h|^3$ on $Q_{T}$, we obtain $(\ref{weak2})$.
	
	\noindent Finally, the uniform convergence of $|h^\kappa|_{\kappa}$ to $|h|$ on $Q_{T}$ implies that $|h^\kappa|_{\kappa}h^\kappa$ converges uniformly to $|h|h$ on $Q_{T}$. Combining this with $(\ref{hx})$ and $(\ref{hxx})$ yields $(\ref{weak1})$. The proof of this lemma is complete.
\end{proof}

By exploiting the singular behavior of the entropy functional as $h \to 0$, we now show that the limit function $h$ remains non-negative and is positive almost everywhere.

\begin{lema}\label{zhengxing}
	If the conditions $(\ref{jiashe})$ hold for the approximate problem $(\ref{a1})-(\ref{a3})$, then the function $h$ in Lemma \ref{shoulianlema} satisfies
	\begin{align}\label{feifu}
		&h \geq 0 \ \ \text{in} \ Q_T,
		\\\label{daoshu}
		&\frac{1}{h}\in L^{\infty}(0,T;L^1(\Omega)).
	\end{align}
	Moreover, the set $\{x \ | \ h(t,x)=0\}$ has measure zero for any $t\in [0,T]$.
\end{lema}
\begin{proof}
	First, we prove that $h \geq 0$ in $Q_T$. Suppose, for the sake of contradiction, that there exists a point $(t_0,x_0) \in Q_T$ such that $h(t_0, x_0) < 0$. By $(\ref{5.3})$, $h^\kappa$ converges uniformly to $h$ on $Q_T$. Thus, there exist $\delta > 0$ and $\kappa_0 > 0$ such that for any $\kappa < \kappa_0$ and all $(t,x)$ satisfying $|t-t_0| + |x-x_0| < \delta$, we have
	\begin{align*}
		h^\kappa(t,x) < -\delta.
	\end{align*}
	For such $(t,x)$, using $(\ref{6.2})$ and the monotone convergence theorem, we obtain
	\begin{align*}
		G_\kappa(h^\kappa) = -\int_{h^\kappa}^{c} g_\kappa(s) \, ds \geq -\int_{-\delta}^{0} g_\kappa(s) \, ds \rightarrow -\int_{-\delta}^{0} g_0(s) \, ds \quad \text{as } \kappa \rightarrow 0,
	\end{align*}
	where $g_0(s) = \lim\limits_{\kappa \rightarrow 0} g_\kappa(s)$. By $(\ref{6.2})$, $g_0(s) = -\infty$ for $s < 0$, so the integral on the right-hand side equals $+\infty$. Thus one conclude that
	\begin{align*}
		\lim_{\kappa \rightarrow 0} \int_{\Omega} G_\kappa(h^\kappa) \, dx = +\infty,
	\end{align*}
	which contradicts $(\ref{6.15})$. This proves $(\ref{feifu})$.
	
	\noindent Next, we show that for any $t \in [0, T]$, the set $\{x \in \Omega \mid h(t,x)=0\}$ has measure zero. If not, then for some $t_1 \in [0,T]$, the set $E := \{x \in \Omega \mid h(t_1,x)=0\}$ has positive measure. From $(\ref{5.3})$, there exists a modulus of continuity $\sigma(\kappa)$ satisfying $0 < \sigma(\kappa) < \frac{c}{2}$ such that for any $x \in E$,
	\begin{align*}
		h^\kappa(t_1,x) < \sigma(\kappa).
	\end{align*}
	For any $\delta > 0$, we have $\sigma(\kappa) < \delta$ provided $\kappa$ is sufficiently small. Then for any $x \in E$,
	\begin{align*}
		G_\kappa(h^\kappa(t_1,x)) \geq -\int_{\sigma(\kappa)}^{c} g_\kappa(s) \, ds \geq -\int_{\delta}^{c} g_\kappa(s) \, ds \rightarrow -\int_{\delta}^{c} g_0(s) \, ds \quad \text{as } \kappa \rightarrow 0.
	\end{align*}
	Applying $(\ref{g})$, we have
	\begin{align*}
		-\int_{\delta}^{c} g_0(s) \, ds \geq \frac{d}{\delta},
	\end{align*}
	where $d > 0$ is a constant. It follows that
	\begin{align*}
		\varlimsup_{\kappa \rightarrow 0} \int_{\Omega} G_\kappa(h^\kappa(t_1,x)) \, dx \geq \frac{d}{\delta} \text{meas}(E) \rightarrow \infty \quad \text{as } \delta \rightarrow 0,
	\end{align*}
	which contradicts $(\ref{6.15})$. Thus, the set $\{x \in \Omega \mid h(t,x)=0\}$ has measure zero for all $t \in [0, T]$.
	
	\noindent Finally, we prove that $1/h \in L^1(\Omega)$ for any $t \in [0,T]$. For points $(t,x)$ where $h(t,x) > 0$, it follows from $(\ref{6.2})$ and $(\ref{5.3})$ that
	\begin{align}\label{6.21}
		G_\kappa(h^\kappa(t,x)) \rightarrow G_0(h(t,x)) \quad \text{a.e. in } Q_T,
	\end{align}
	where $G_0$ satisfies $(\ref{g})$. Since $\{x \in \Omega \mid h(t,x)=0\}$ has measure zero for any $t\in [0,T]$, it follows that, $(\ref{6.21})$ holds for almost all $x$ for each $t \in [0,T]$. Using $(\ref{6.15})$ and Fatou's lemma, we deduce that for all $t \in [0,T]$,
	\begin{align*}
		\int_{\Omega} G_0(h(t,x)) \, dx \leq C.
	\end{align*}
	In view of $(\ref{g})$, this yields $1/h \in L^1(\Omega)$ for any $t \in [0,T]$, which proves $(\ref{daoshu})$. The proof of this lemma is complete.
\end{proof}

\noindent{\textbf{Proof of Theorem 1.2.}} Let $h^\kappa$ be the sequence of solutions to the approximate problem $(\ref{a1})-(\ref{a3})$.
According to Lemma $\ref{shoulianlema}$, there exists a subsequence, still denoted by $h^\kappa$, that converges to $h$. We shall show that $h$ is a global weak solution of the original problem $(\ref{1})-(\ref{3})$ in sense of Definition $\ref{dingyi1}$.

Recalling $(\ref{5.1})$ yields $(\ref{d1})$. To verify $(\ref{d2})$, we multiply $(\ref{a1})$ by a test function $\varphi \in C^\infty([0,T] \times \bar{\Omega})$ satisfying $\varphi = 0$ on $[0,T] \times \partial\Omega$ and $\varphi(T, \cdot) = 0$. Integrating the resulting equation over $Q_{T}$ and using $(\ref{a2})-(\ref{a3})$, we obtain
\begin{align*}
	&(h^\kappa,\varphi_t)_{Q_{T}}-\alpha_{1} ( h_{xx}^\kappa, (|h^\kappa|_{\kappa}^3\varphi_x)_x)_{Q_{T}} 
	\nonumber\\
	=&\alpha_{2} (|h^\kappa|_{\kappa}^3 h_x^\kappa,  \varphi_x)_{Q_{T}} - \alpha_{3} (|h^\kappa|_{\kappa}^2 h_x^\kappa, \varphi_x)_{Q_{T}}-(h_0,\varphi(0))_{\Omega}.
\end{align*}
Equation $(\ref{d2})$ follows from this relation provided that we prove that
\begin{align} \label{5.29}
	(h^\kappa,\varphi_t)_{Q_{T}} \quad &\rightarrow \quad (h,\varphi_t)_{Q_{T}},
	\\ \label{5.30}
	( h_{xx}^\kappa, (|h^\kappa|_{\kappa}^3\varphi_x)_x)_{Q_{T}} \quad &\rightarrow \quad ( h_{xx},(h^3\varphi_x)_x)_{Q_{T}} ,
	\\ \label{5.31}
	(|h^\kappa|_{\kappa}^3 h_x^\kappa, \varphi_x)_{Q_{T}} \quad &\rightarrow \quad (h^3 h_x, \varphi_x)_{Q_{T}},
	\\ \label{5.32}
	(|h^\kappa|_{\kappa}^2 h_x^\kappa, \varphi_x)_{Q_{T}} \quad &\rightarrow \quad (h^2 h_x, \varphi_x)_{Q_{T}}.
\end{align}
for $\kappa\rightarrow 0$. Specifically, $(\ref{5.29})$ is a direct consequence of $(\ref{5.3})$, while $(\ref{5.31})$ follows from $(\ref{5.4})$ and $(\ref{feifu})$. Combining $(\ref{5.5})$ with $(\ref{feifu})$, we see that $(\ref{5.32})$ holds. Furthermore, by applying $(\ref{weak2})$, $(\ref{weak1})$, $(\ref{feifu})$, and the expansion $(|h^\kappa|_{\kappa}^3 \varphi_{x})_{x} = 3|h^\kappa|_{\kappa} h^\kappa h_{x}^\kappa \varphi_{x} + |h^\kappa|_{\kappa}^3 \varphi_{xx}$, we obtain $(\ref{5.30})$.
Thus, we conclude that
\begin{align}\label{zhengzeruojie}
	&(h, \varphi_t)_{Q_{T}} - \alpha_{1} (h_{xx}, (h^3 \varphi_x)_{x})_{Q_{T}} 
	\nonumber\\
	= &\alpha_{2} (h^3 h_x, \varphi_x)_{Q_{T}} - \alpha_{3} (h^2 h_x, \varphi_x)_{Q_{T}} - (h_0, \varphi(0))_{\Omega}.
\end{align}
By combining $(\ref{5.1})$, $(\ref{feifu})$, and $(\ref{zhengzeruojie})$, we establish the existence of a nonnegative weak solution to problem $(\ref{1})-(\ref{3})$ in the sense of Definition $\ref{dingyi1}$. Moreover, $(\ref{5.2})$ ensures that $(\ref{2.9})$ holds, and $(\ref{daoshujieguo})$ follows from $(\ref{daoshu})$. Finally, Lemma $\ref{zhengxing}$ implies $(\ref{jieguozheng})$ and for any $t \in [0,T]$, the set $\{x \mid h(t,x)=0\}$ is of measure zero. This completes the proof of the theorem.

\vskip0.5cm

\noindent\textbf{Data availability}: No data was used for the research described in the article.

\noindent\textbf{Confict of interest}: The authors declare that they have no confict of interest.

\end{document}